\documentclass[a4paper]{amsart}
\usepackage{enumerate}
\usepackage{amssymb}
\usepackage{stmaryrd}
\newtheorem{thm}{Theorem}[section]

\newtheorem{prop}[thm]{Proposition}

\theoremstyle{definition}
\newtheorem{defn}[thm]{Definition}
\newtheorem{ex}[thm]{Example}
\newtheorem{exs}[thm]{Examples}

\theoremstyle{remark}
\newtheorem{rem}[thm]{Remark}
\newtheorem{rems}[thm]{Remarks}

\newcommand{\brr}[1]{\left[#1\right]}
\newcommand{\brg}[1]{\llbracket#1\rrbracket}

\newcommand{\ZZ}{\mathbb{Z}}
\newcommand{\ds}{\displaystyle}

\newcommand{\diff}{\rm{d}}

\newcommand{\al}{\alpha}
\newcommand{\be}{\beta}

\newcommand{\A}{A}                      
\renewcommand{\gg}{\mathfrak{g}}        
\newcommand{\hh}{\mathfrak{h}}          

\DeclareMathOperator{\ad}{ad}           
\DeclareMathOperator{\End}{End}         

\usepackage[dvips]{graphicx}

\begin{document}

\title{Hom-Lie algebroids}

\author{Camille Laurent-Gengoux
and Joana Teles }
\thanks{Lab. de Math. et Appl., UMR 7122,
Universit\'e de Lorraine, \^Ile du Saulcy
57045, Metz,
France \emph{ and} CMUC, Department of Mathematics, University of Coimbra, 3001-454 Coimbra, Portugal}
\thanks{Both authors where supported by the Centre for Mathematics of the
University of Coimbra and Funda\c{c}\~{a}o para a Ci\^{e}ncia e a Tecnologia,
through European program COMPETE/FEDER and by FCT (Funda\c{c}\~{a}o para a Ci\^{e}ncia e a Tecnologia) grant PTDC/MAT/099880/2008.}

\begin{abstract}
We define hom-Lie algebroids, a definition that may seem cumbersome at first, but which is justified, first, by a one-to-one corespondence with hom-Gerstenhaber algebras, a notion that we also introduce, and several examples, including hom-Poisson structures.

\end{abstract}

\maketitle

\section{Introduction}

There is an increasing interest for hom-structures, which are, to make a long story short,
Lie algebra-like structures, equipped with an additional map, say $\alpha$, and in which the cubic identities (e.g. the associativity or
the Jacobi conditions) are replaced by a relation of the same form, in which the first variable, say $x$,
is replaced by $\alpha (x)$, so that the henceforth obtained condition is
always true when restricted to the kernel of that map.
Originally, the study of hom-Lie algebras, initiated by  \cite{HLS} showed a
natural occurrence of this notion while studying cocycles of the Virasoro algebra.
In the following years, Makhlouf, Silvestrov and their coauthors \cite{MS} have
showed that several classical algebraic structures
admit natural generalizations when, instead of just a vector space,
we start with a vector space and an automorphism of it, leading to investigate hom-associative algebras \cite{MS},
hom-Jordan algebras \cite{MS2010}, admissible algebras \cite{MS2009}, hom-Poisson algebras \cite{MS2}, to cite a few.
The interest in hom-Poisson structure is likely to grow again due to the recent thesis of Olivier Elchinger \cite{E}, who introduces quantization of Hom-Poisson structures, giving, in particular, explicit formulas for the Moyal product, since this raises the question of integration of hom-Lie algebroids.

Our purpose is to introduce hom-Lie algebroids.
We would like to insist that it is not straightforward at all to see what this definition should be.
This should not come as a surprise: a definitive notion of Hom-group, allowing to state Lie I, II and III theorems,
is still to be found.
In particular, there is no such a thing as a hom-Lie groupoid that could give us a hint.
To derive a definition that makes sense, we indeed had to go through the notion of hom-Gerstenhaber algebra, but even there
there was an unexpected phenomenon,  a hom-Gerstenhaber algebra is not  hom-associative,
as one could have expected, hence defining
a hom-Lie algebroid does not reduce simply adding the prefix hom- to classical definitions and results in a systematic manner.

We would like to thank Raquel Caseiro for several interesting discussions.

\section{Hom-Lie algebras and hom-Poisson algebras}

Given $\gg$ a vector space and a bilinear map $ \brr{\, , \, }: \gg \otimes \gg \to \gg$, we call \textbf{automorphism of
$(\gg,\brr{\, , \, } )$} a linear map $\al: \gg \to \gg$ such that
$$
\al (\brr{x,y}) = \brr{\al (x), \al (y)}
$$
for all $x,y \in \gg$.

\begin{defn}\label{def:hom:Lie:algebra} \cite{HLS}
A \textbf{hom-Lie algebra} is a triple  $(\gg, \brr{\, , \, }, \al)$ with $\gg$ a vector space equipped with a
skew-symmetric bilinear map $ \brr{\, , \, }:\gg \otimes \gg \to \gg$ and an automorphism
$\al$ of $(\gg,\brr{\, , \, })$ such that:
\begin{equation}
\brr{\al (x),\brr{y,z}}+\brr{\al (y),\brr{z,x}}+\brr{\al
(z),\brr{x,y}}=0, \quad \forall x,y,z\in\gg  \quad \hbox{(hom-Jacobi identity)}. \label{eq:hom:Jacobi:algebra}
\end{equation}
A \textbf{morphism} between hom-Lie algebras $(\gg,\brr{\, , \, }_\gg,\al)$ and $(\hh,\brr{\, , \, }_\hh,\be)$ is a linear map $\psi:\gg \to \hh$
such that $\ds \psi(\brr{x,y}_\gg)=\brr{\psi( x),\psi (y)}_\hh$ and $\psi (\al(x))=\be(\psi(x))$
for all $x,y \in \gg$.
When $\hh$ is a vector subspace of $\gg$ and $\psi$ is the inclusion map, one speaks of \textbf{hom-Lie subalgebra}.
\end{defn}

In a similar fashion, one defines \textbf{graded hom-Lie algebras} to be triples $(\gg, \brr{\, , \, }, \al)$ with $\gg=\ds{ \oplus_{i \in \ZZ} \gg_i}$
a graded vector space, $ \brr{\, , \, }:\gg \otimes \gg \to \gg$
a graded skew-symmetric bilinear map of degree $0$ and $\al : \gg \to \gg$ an automorphism of  $(\gg,\brr{\, , \, })$ of degree $0$ satisfying
for all $x \in \gg_i,y \in \gg_j,z\in\gg_k$:
\begin{eqnarray}\label{eq:hom:graded:Jacobi:algebra}
\hbox{\small{$(-1)^{ik}$}} \brr{\al (x),\brr{y,z}}+ \hbox{\small{$(-1)^{ji}$}} \brr{\al (y),\brr{z,x}}+\hbox{\small{$(-1)^{kj}$}}  \brr{\al (z),\brr{x,y}}=0,   \quad \hbox{(graded hom-Jacobi identity)}. \nonumber
\end{eqnarray}

Of course, these definitions make sense for finite dimensional or infinite dimensional vector spaces indifferently.

\begin{ex}\label{ex:composition}
(See e.g. \cite{MS}).
Given a vector space $\gg$ equipped with a skew-symmetric bilinear map $ \brr{\, , \, }:\gg \otimes \gg
\to \gg$ and an automorphism $\al:\gg\to \gg$ of $(\gg, \brr{\, , \, })$, define
$ \brr{\, , \, }_\al:\gg \otimes \gg
\to \gg$ by
$$ \brr{x,y}_\al=\al (\brr{x,y}), \quad \hbox{ $\forall x,y \in \gg$.} $$
 Then $(\gg, \brr{\, , \, }_{\al}, \al)$ is a hom-Lie algebra if and only if the restriction of $\brr{\, , \, }$
to the image of $\al^2 $ is a Lie bracket.
In particular, hom-Lie structures are naturally associated to Lie algebras equipped with a Lie algebra automorphism \cite{Yau09}.
Such hom-Lie structures are said to be \textbf{obtained by composition}.
\end{ex}

\begin{defn} \label{hom-associative}\cite{MS}
A \textbf{hom-associative algebra} is a triple $(\A, \mu, \al)$ consisting of a  vector space $\A$, a bilinear map $\mu: \A \otimes \A \to \A$ and an automorphism $\al$ of $( \A, \mu) $ satisfying
$$
\mu(\al(x), \mu(y,z))= \mu (\mu(x,y), \al (z)), \quad \forall x,y,z \in \A \hbox{ (hom-associativity)}.
$$
\end{defn}

\begin{ex}\cite{Yau09}
As in example \ref{ex:composition}, given $(\A, \mu)$ an associative algebra and $\al: \A \to \A$ an algebra automorphism, the
triple $(\A, \mu_{\al}:= \al \circ \mu , \al)$ is a hom-associative algebra, said again to be \textbf{obtained by composition}.
\end{ex}

As one can expect, the commutator of a hom-associative algebra is a hom-Lie algebra:

\begin{ex} \label{ex:composition2}
\cite{MS} For every hom-associative algebra $(\A, \mu, \al)$ (see definition \ref{hom-associative} above),
the triple $(\A, \brr{\, , \, }, \al)$  is a a hom-Lie algebra, where
$$
\brr{x,y}:= \mu(x,y)- \mu (y,x)
$$
for all $x,y \in \A$.
\end{ex}

A Poisson algebra being a space endowed with an associative and a Lie product, satisfying some compatibility relation,
the next definition is perfectly natural:

\begin{defn} \cite{MS2}
A \textbf{hom-Poisson algebra} is a quadruple $(\A, \mu, \{ \, ,\, \}, \al)$ consisting of a vector space $\A$,  bilinear maps $\mu: \A \otimes \A \to \A$ and $ \{ \, , \, \} : \A \otimes \A \to \A$ and a linear map $\al: \A \to \A$ such that:
\begin{enumerate}
\item $(\A, \mu, \al)$ is a commutative hom-associative algebra,
\item $(\A, \{ \,,\, \}, \al)$ is a hom-Lie algebra,
\item $ \{ \al(x), \mu(y,z) \} = \mu (\al (y), \{x, z\} ) + \mu ( \{ x, y \}, \al(z))$, for all $ x,y,z \in \A$.
\end{enumerate}
\end{defn}

There is, however, a slightly related notion, that shall be useful in the sequel:

\begin{defn}
A \textbf{purely hom-Poisson algebra} is a quadruple $(\A, \mu, \{ \, ,\, \}, \al)$ consisting of a vector space $\A$,  bilinear maps $\mu: \A \otimes \A \to \A$ and $ \{ \, , \, \} : \A \otimes \A \to \A$ and a linear map $\al: \A \to \A$ such that:
\begin{enumerate}
\item $(\A, \mu)$ is a commutative associative algebra,
\item $(\A, \{ \,,\, \}, \al)$ is a hom-Lie algebra,
\item $ \{x, \mu(y,z) \} = \mu (\al (y), \{x, z\} ) + \mu ( \{ x, y \}, \al(z))$, for all $ x,y,z \in \A$.
\end{enumerate}
\end{defn}


\begin{ex}\label{Poisson} \cite{Yau10}
Let $(\A, \mu , \{ \, ,\, \})$ be a Poisson algebra and
$\al : \A \to \A$ a Poisson automorphism, then
the quadruple $(A, \mu_{\al}:= \al \circ \mu , \{\, ,\, \}_{\al}:= \al \circ \{ \, ,\, \}, \al)$ is a hom-Poisson algebra,
said to be \textbf{obtained by composition}. It is indeed enough to assume that $\{ \, ,\, \}$ (resp. $\mu$) is a Lie
bracket (resp. an associative product) when restricted to the image of $\al^2$.

Also,  $(A, \mu, \{\, ,\, \}_{\al}:= \al \circ \{ \, ,\, \}, \al)$ is a purely hom-Poisson algebra.
\end{ex}

\begin{ex}\label{Poisson2}
In particular, given $(M, \pi)$ a manifold equipped with a  bivector field $ \pi$, and $\varphi: M \to M$ a smooth map,
then a hom-Poisson structure on $C^\infty(M)$ can be obtained by composition provided that $\varphi$ preserves the bivector field $ \pi $
(i.e. $\pi_{\varphi(m)} = (\wedge^2 T_m \varphi) (\pi_m)$ for all $m \in M$) and that the Schouten-Nijenhuis bracket  $[\pi,\pi]$  is a trivector field that vanishes on $\varphi^2 (M) \subset M$.
Under these conditions, $(C^\infty(M), \mu_{\varphi^*}, \{ \, ,\, \}_{\varphi^*}= \varphi^* \circ \{ \, ,\, \}, \varphi^*)$ is a hom-Poisson algebra where $\mu$ is the usual product on $C^\infty(M)$ (and $(C^\infty(M), \mu, \{ \, ,\, \}_{\varphi^*}, \varphi^*)$ is a purely hom-Poisson algebra).
\end{ex}

Example \ref{Poisson2} makes the following definition natural: a triple $(M,\pi,\varphi)$,
with $\pi$ a bivector field on a manifold $M$ and $\varphi: M \to M$ a smooth map,
is called a \textbf{hom-Poisson manifold} when $\varphi$ preserves the bivector field $ \pi $
 and that the Schouten-Nijenhuis bracket  and $[\pi,\pi]$ vanishes on $\varphi^2 (M) \subset M$.

\begin{ex}\label{ex:dualhomLiealgebra}
Here are examples of hom-Poisson algebras that are not obtained by composition in general, see \cite{BEM} for an alternative description.
Let $(\gg,\brr{\, , \, }, \al)$ be a hom-Lie algebra.
Equip its symmetric algebra $S(\gg)$ with the product $\mu_{\al}:={\al} \circ \mu $, where $\mu (x,y) = x \odot y$ is the symmetric product
and ${\al}: S(\gg) \to S(\gg) $ stands for the automorphism of $(S(\gg),\mu)$ given by $
{\al} (x_1 \odot \cdots \odot x_n):= \al (x_1) \odot \cdots \odot \al (x_n)$ for all $x_1, \dots, x_n \in \gg$. The quadruple $(S(\gg), \mu_{\al} , \{\, , \, \}, {\al})$ is a hom-Poisson algebra where
$$
\{x_1 \odot\dots  \dots \odot x_p ,y_1 \odot \dots \odot y_q\} := \sum_{i=1}^p \sum_{j=1}^q [x_i, y_j] \odot \al \left(x_1 \odot\dots \widehat{x}_i  \odot \dots \odot x_p \odot y_1 \odot \dots  \widehat{y}_j \odot \dots \odot y_q\right),  
$$
for all $x_1,\dots,x_p,y_1, \dots,y_q \in \gg$.
Identifying $S(\gg)$ with polynomial functions on $\gg^*$, we could also write:
      $$
      \{ F, G \} (a)= \left\langle \brr{ \left. \diff F\right|_{\alpha^* (a)},\left. \diff G\right|_{\alpha^* (a)} }, a \right\rangle
      $$
for all polynomial functions  $F,G $ on $\gg^*$ and all $a \in \gg^*$, with the understanding that the differential of a function of
$\gg^*$, a priori an element in $T^* \gg^*$ is considered as an element in $\gg$.
It deserves to be noticed that $\gg^*$ is \emph{not} a hom-Poisson manifold in general.
Also, it is not clear how we can associate a purely hom-Poisson algebra structure on $S(\gg)$.
\end{ex}

\section{Hom-Gerstenhaber algebras}

Lie algebroids structures on a vector bundle $A \to M$ are in one-to-one correspondence with Gerstenhaber algebra structures
on $\Gamma(\wedge^\bullet A [-1])$, see e.g. \cite{KSM1, KSM2, McKX}, making natural the idea of defining
hom-Lie algebroids through the following object:

\begin{defn}
A \textbf{hom-Gerstenhaber algebra} is a quadruple $({\mathcal A} = \oplus_{i \in {\mathbb Z}} \mathcal{A}_i , \wedge, \brg{\, ,\, }, \al)$ where  $({\mathcal A} = \oplus_{i \in {\mathbb N}} \mathcal{A}_i , \wedge)$ is  a graded commutative associative algebra, $\al$ is an automorphism of $({\mathcal A}  , \wedge)$ of degree $0$ and $\brg{\, , \, }: {\mathcal A} \otimes {\mathcal A} \to {\mathcal A}$ is a bilinear map of degree $-1$
such that:
\begin{enumerate}
  \item $(\mathcal{A}[1],\brg{\, , \, },\al)$ is a graded hom-Lie algebra
(as usual, $ \mathcal{A}[1]$ refers to the graded vector space whose component of degree $i$ is ${\mathcal A}_{i+1}$, for all $i \in {\mathbb Z}$);
  \item the \textbf{hom-Leibniz rule} holds:
    $$
    \brg{X, Y \wedge Z} = \brg{X,Y} \wedge \al(Z) + \hbox{\small{$(-1)^{(i-1) j}$}} \al(Y) \wedge \brg{X,Z}, 
  $$
  for all $X \in \mathcal{A}_i,Y \in \mathcal{A}_j,Z \in  \mathcal{A}$.
\end{enumerate}
\end{defn}

\begin{rem}
Notice that $\wedge$ is assumed to be an associative product, not a hom-associative product,
so that a hom-Gerstenhaber algebra is not an odd version of a hom-Poisson algebra.
But it might be seen as an odd version of a purely hom-Poisson algebra.
\end{rem}

\begin{ex}\label{ex:GerstByCompo}
Given an automorphism $\al$ of a Gerstenhaber algebra $({\mathcal A}  , \wedge, \brg{\, ,\, })$
(i.e. $\al$ is an automophism for both pairs $({\mathcal A}, \wedge)$ and  $({\mathcal A}, \brg{\, ,\, })$), then
$({\mathcal A} , \wedge, \alpha \circ \brg{\, ,\, },\al)$ is a hom-Gerstenhaber algebra, said to be \textbf{obtained by composition}.
Again, it suffices to assume that $ \brg{\, ,\, }$ satisfies the Jacobi identity on the image of $\al^2$.
\end{ex}

\begin{ex}\label{ex:gerstenhaber0}
Let $\gg$ be a vector space equipped with a skew-symmetric bilinear map $ \brr{\, , \, }:\gg \otimes \gg \to \gg$ and  an automorphism $\al$ of $(\gg , \brr{\, , \, })$.
The triple $(\gg, \brr{\, , \, }, \al)$ is a hom-Lie algebra if and only if
the quadruple $(\wedge^\bullet \gg, \wedge,\brg{\, , \,},\al)$ is a  hom-Gerstenhaber algebra where
$$
\brg{x_1 \wedge \dots \wedge x_p, y_1 \wedge \dots \wedge y_q}= \sum_{i=1}^p \sum_{j=1}^q \hbox{\small{$(-1)^{i+j}$}}    \brr{x_i, y_j} \wedge \al(x_1 \wedge \dots \widehat{x}_i  \wedge \dots \wedge x_p \wedge y_1 \wedge \dots  \widehat{y}_j \wedge \dots \wedge y_q)
$$
for all $x_1 , \dots , x_p, y_1 ,\dots , y_q \in \gg$ and
$${\al}(x_1 \wedge \dots \wedge x_p) =  \al(x_1) \wedge \dots \wedge \al(x_p).
$$
The only difficult part is to prove that the graded hom-Jacobi identity holds on $\wedge^\bullet \gg$
if and only if $\brr{\, , \,}$ satisfies the hom-Jacobi identity. This follows from the fact (which follows
from a cumbersome but direct computation) that the hom-Jacobiator, defined as
$$
Jac_{\al}(X,Y,Z):= \hbox{\small{$(-1)^{(i-1)(k-1)}$}} \brg{ \al(X),\brg{Y,Z}} +
\hbox{\small{$(-1)^{(j-1)(i-1)} $}} \brg{\al(Y), \brg{Z,X}}+
\hbox{\small{$(-1)^{(k-1)(j-1)} $}} \brg{\al(Z), \brr{X,Y}}
$$
for all  $X \in \wedge^i \gg,Y  \in \wedge^j \gg,Z \in \wedge^k \gg $, satisfies
\begin{equation}\label{eq:Jacobiator}
Jac_{\al}(XY,Z,T) = \al^2(X) Jac_{\al}(Y,Z,T)  +  \hbox{\small{$(-1)^{ij} $}}\al^2(Y) Jac_{\al}(X,Z,T)
\end{equation}
and is a graded skew-symmetric map, so that it vanishes if and only if its restriction to $\wedge^0 \gg = {\mathbb R} $ and $ \wedge^1 \gg = \gg$ vanishes.
\end{ex}

Let $(\gg, \brr{\, , \, },
\al)$ be a hom-Lie algebra, denote by $\al^s$ the $s$-power of $\al$, $s \geq 1$, i.e.
$$
\al^s = \al \circ \cdots \circ \al \quad (s \mbox{ times}).
$$

For any element $x$ in the hom-Lie algebra $(\gg, \brr{\, , \, },
\al)$ define the $\al^s$-adjoint map $\ad_x^s:\gg \longrightarrow \gg$ by
$\ad_x^s y = \brr{\al^s(x),y}$. Using the hom-Jacobi identity
\eqref{eq:hom:Jacobi:algebra} we obtain
\begin{equation}\label{eq:adjoint}
\ad_{\brr{x,y}}^s \circ \al= \ad_{\al (x)}^s \circ \ad^s_{y} - \ad^s_{\al
(y)} \circ \ad^s_{x}.
\end{equation}

Let us recall the following definition:
\begin{defn}\label{def:hom:Lie:algebra:representation}  \cite{Sh}
A \textbf{representation} of a hom-Lie algebra $(\gg, \brr{\, , \, }, \al)$ on a vector space $V$
is a pair $(\rho,\al_V)$ of linear maps
 $\rho: \gg \to \mathfrak{gl}(V)$, $\al_V: V  \to V$ such that: 
\begin{eqnarray}
&& \rho(\al(x)) \circ \al_V  =  \al_V  \circ \rho(x) \label{rho:A} \\
&& \rho(\brr{ x,y}) \circ  \al_V  = \rho(\al(x)) \circ \rho(y) - \rho (\al(y)) \circ \rho(x), \label{rho:Jacobi}
\end{eqnarray}
for all $x,y \in \gg$.
\end{defn}

\begin{exs}

{\bf a.} \cite{Sh} The $\al^s$-adjoint map defines a representation $(ad^s, \al)$ of a hom-Lie algebra  $(\gg, \brr{\, , \, },
\al)$ on the vector space $\gg$.

\noindent
 {\bf b.} Let $(\rho, \al_V )$ be a representation of the hom-Lie algebra  $(\gg, \brr{\, , \, },
\al)$ on the vector space $V$. Then $(\wedge^\bullet \gg \otimes S^\bullet (V), \wedge , \brg{\, , \,}, \al )$
is a hom-Gerstenhaber algebra  where
$\al :  \wedge^\bullet \gg \otimes S^\bullet(V) \to \wedge^\bullet \gg \otimes S^\bullet(V)$ is defined as
$\al (x_1 \wedge \dots \wedge x_p \otimes v_1 \odot \dots \odot v_q)= \al(x_1) \wedge \dots \wedge \al(x_p) \otimes \al_V (v_1) \odot \dots \odot \al_V (v_q)
$,
the bracket $\brg{\, ,\, }$  being the hom-Gerstenhaber bracket whose restriction to $\wedge^0 \gg \otimes S(V) $ vanishes,
whose restriction to  $\wedge^\bullet \gg \otimes S^0 (V) \simeq  \wedge^\bullet \gg $ is as in example \ref{ex:gerstenhaber0}, and such that:
$$
\brg{x,  v_1 \odot \dots \odot v_q}= \sum_{i=1}^q \rho(x)(v_i)\odot {\al}( v_1 \odot \dots \odot\hat{v}_i \odot \dots \odot v_q),
$$
for all $x \in \gg$ and $ v_1, \dots, v_q \in V$.
\end{exs}

\begin{defn}
Given an algebra automorphism $\al$ of a commutative associative algebra $A$, we call \textbf{$\al$-derivation} a map $\delta:A \to A$
which satisfies
$$
\delta (FG) = \al(F)  \delta (G) +  \al(G)  \delta (F)
$$
for all $F,G \in A$.

We denote by $der_{\al}(A)$ the set of all $\al$-derivations.
\end{defn}

\begin{rem}
Let $M$ be a manifold and $\varphi:M \to M$ a smooth map. Then,
$der_{\varphi^*}(C^\infty (M))$ can be identified with $ \Gamma(\varphi^! TM)$,
by mapping a section $ X \in \Gamma(\varphi^! TM)$
to the $\varphi^*$-derivation mapping a function $F $ to the function whose value at $m \in M $
 is $ X_m \diff_{\varphi (m)} F $, with the understanding that $ X_m$ must be considered as an element in
 $T_{\varphi (m)} M$.
 \end{rem}

\begin{prop}\label{prop:gerstenhabergivesanchor}
For every  hom-Gerstenhaber algebra  $({\mathcal A} = \oplus_{i \in {\mathbb N}} \mathcal{A}_i , \wedge, \brg{\, , \, }, \al)$,
denote by  $\rho : \mathcal{A}_1 \longrightarrow \End (\mathcal{A}_0)$ the map given by
$ \rho(X) [F] : = \brg{X,F} $
for all $F \in  {\mathcal{A}_0}$.
Then $\mathcal{A}_0$ is a commutative associative algebra, $\left. \al\right|_{\mathcal{A}_0}$ is an algebra
automorphism of $ \mathcal{A}_0$, $F \mapsto \rho(X)(F)$ is,  for all $X \in {\mathcal A}_1$, a $ \left. \al\right|_{\mathcal{A}_0}$-derivation of
$ \mathcal{A}_0$,
the triple $(\mathcal{A}_1, \left. \brg{ \, , \, }\right|_{\mathcal{A}_1 \times \mathcal{A}_1}, \left. \al\right|_{\mathcal{A}_1} )$
is a hom-Lie algebra,
and  $(\rho,\left. \al\right|_{\mathcal{A}_0})$ is a representation of $(\mathcal{A}_1, \left. \brg{ \, , \, }\right|_{\mathcal{A}_1 \times \mathcal{A}_1}, \left. \al\right|_{\mathcal{A}_1} )$ on  $\mathcal{A}_0$.
\end{prop}

\begin{proof}
Only the last point needs justification. We recover relations (\ref{rho:A}) and (\ref{rho:Jacobi}) by using that $\al$ is an automorphism of $\brg{\, , \,}$ of degree $0$ and   the graded hom-Jacobi identity
as follows,
\begin{eqnarray*}
&& \left. \al\right|_{\mathcal{A}_0}  (\rho(X) [F] )= \rho \left( \left. \al\right|_{\mathcal{A}_1}(X) \right) \left[ \left.  \al\right|_{\mathcal{A}_0}(F) \right] \\
&& \rho(\brg{X,Y}) [\left. \al\right|_{\mathcal{A}_0}(F)] = \left( \rho ( \left. \al\right|_{\mathcal{A}_1}(X) ) \circ \rho(Y)  - \rho(\left. \al\right|_{\mathcal{A}_1}(Y)) \circ \rho(X) \right) [F],
\end{eqnarray*}
 for all $X,Y \in \mathcal{A}_1$ and $F \in \mathcal{A}_0$.
\end{proof}

\section{Definition of hom-Lie algebroid }

We can now, at last, define hom-Lie algebroids:

\begin{defn}\label{def:homLiealgebroid}
\textbf{A hom-Lie algebroid} is a quintuple $( A \to M, \varphi,  \brr{\, , \,},\rho,\al) $, where
$A \to M$ is a vector bundle over a manifold $M$, $\varphi:M \to M$ is a smooth map,
$\brr{\,  , \,}: \Gamma(A) \otimes \Gamma(A) \to \Gamma(A)$ is a bilinear map, called \textbf{bracket},
$\rho : \varphi^! A \to \varphi^! TM $ is a vector bundle morphism, called \textbf{anchor}, and
$\al:\Gamma(A) \to \Gamma(A) $ is a linear endomorphism  of $\Gamma(A) $  such that
\begin{enumerate}
  \item[1.]  $\al(FX) = \varphi^* (F) \al (X)$, 
for all $X \in \Gamma(A), F \in C^\infty(M)$;
  \item [2.]
the triple $(\Gamma(A),\brr{\,  , \,}, \al)$ is a hom-Lie algebra;
   \item[3.] the following hom-Leibniz identity holds:
 $$ [X,FY]= \varphi^*(F) [X,Y] + \rho(X) [F] \, \al(Y), \quad \forall X, Y \in \Gamma(A), F \in C^\infty(M). $$
    \item[4.] $(\rho,\varphi^*)$ is a representation of $(\Gamma(A),\brr{\,  , \,}, \al)$ on $C^\infty(M)$.
\end{enumerate}
\end{defn}

\begin{rems}

{\bf a.} Linear endomorphisms  of $\Gamma(A) $,   $\al:\Gamma(A) \to \Gamma(A) $, satisfying $\al(FX) = \varphi^* (F) \al (X)$
for all $X \in \Gamma(A), F \in C^\infty(M)$ are in one-to-one correspondence with 
vector bundle morphisms from $\varphi^! A$ to $ A$ over the identity of $ M$.
Given $X \in \Gamma(A)$, a section of the pull-back bundle $\varphi^! A$ is given by
mapping $m \in M$ to $ X_{\varphi(m)} \in A_{\varphi(m)} \simeq (\varphi^! A)_m $. Applying a vector bundle morphism
from $\varphi^! A$ to $ A$ over the identity of $M$
to that section yields a section of $A$, and the henceforth defined assignment $\al$ satisfies $\al(FX) = \varphi^* (F) \al (X)$, 
for all $X \in \Gamma(A), F \in C^\infty(M)$. Moreover, every endomorphism of $\Gamma(A)$ satisfying this relation is of that form.

\noindent
{\bf b.} The hom-Leibniz identity implies that, given sections $X,Y$ of $A$, the value of $ \brr{X  , Y}$ at a given point $m \in M$ depends
only on the first jet of $X$ and $Y$ at $\varphi(m)$.

\noindent
{\bf c.} Above, $\rho(X) [F] $ stands for the function on $M$ whose value at $m \in M$
is $\langle \diff_{\varphi(m)} F  , \rho_m(X_{\varphi(m)}) \rangle $ where $\rho_m : (\varphi^! A)_m \simeq A_{\varphi(m)} \to
 (\varphi^! TM)_m \simeq T_{\varphi(m)}M$ is the anchor map evaluated at $m \in M$ and
 $X_{\varphi(m)}$ is the value of the section $X \in \Gamma(A)$ at $\varphi(m) \in M$.
\end{rems}



\begin{ex}
When $\alpha$ (hence $\varphi$) is the identity map, a hom-Lie algebroid $( A \to M, \varphi, \al, \brr{\, , \,},\rho)  $ is simply a Lie algebroid
\cite{McK}.
 A hom-Lie algebra $(\gg,\brr{\, ,\,}, \al)$ is a hom-Lie algebroid over a singleton.
More generally, define an \textbf{action of a hom-Lie algebra}  $(\gg, \brr{\, , \, }, \al)$ on the manifold $M$, equipped with a smooth map $\varphi: M \to M$,
 to be a linear map $\delta $ from $\gg$ to the space of $\varphi^*$-derivations such that
 $(\delta,\varphi^*)$ defines a representation of the hom-Lie algebra $(\gg, \brr{\, , \, },
\al)$ on the vector space $C^\infty(M)$.
Then a hom-Lie algebroid is obtained by considering the trivial vector bundle  $A=M \times \gg \to M$, the linear map $\al_{A}$ mapping $Fc_v \to \varphi^* (F) c_{\al (v)}$
for all $v \in \gg, F \in C^{\infty}(M)$, the anchor $\rho $ mapping $A_{\varphi(m)} \simeq \gg $
to the element of $T_{\varphi (m)}M$ given by the pointwise derivation $F \mapsto \left. \delta(v)[F]\right|_{m}$ and the bracket
 given by:
 $$
 [F c_v , G c_w]= \varphi^*(FG)  c_{[v,w]}+ \varphi^*(F) \rho(v)[G] c_{\al(w)} - \varphi^* (G) \rho(w)[F]  c_{\al (v)}
 $$
for all $F,G \in C^\infty(M)$, $ v,w \in \gg$. In the previous, $c_v,c_w$ denote the constant sections of $M \times \gg \to M $ given by $m \mapsto (v,m) $ and $m \mapsto (w,m) $ respectively. This hom-Lie algebroid is not obtained by composition in general.
\end{ex}

The following theorem is a consequence of proposition \ref{prop:gerstenhabergivesanchor}, and will allow us to give more examples.

\begin{thm}\label{theo:maintheorem}
Let $A \to M$ be a vector bundle, $\varphi: M \to M$ a smooth map, $\al: \Gamma(A) \to \Gamma(A)$ a linear endomorphism
satisfying $\alpha(FX) = \varphi^* (F) \alpha (X)$, 
for all $X \in \Gamma(A), F \in C^\infty(M)$.
Denote by $\al $ again its extension to $\al: \Gamma(\wedge^\bullet A) \to \Gamma(\wedge^\bullet A)$ given by:
\begin{equation}
\label{eq:alphaextented}
\al( F X_1 \wedge \dots \wedge X_p ) = \varphi^* (F)  \al(X_1) \wedge \dots \wedge \al(X_p)
\end{equation}
for all $p \in {\mathbb N}$, $X_1,\dots, X_p \in \Gamma(A)$, $F \in C^{\infty}(M)$.
Then there is a one-to-one correspondence between hom-Gerstenhaber algebra
structures on $(\Gamma(\wedge^\bullet A), \wedge, \brg{\, ,\,} ,\alpha) $ and hom-Lie algebroids structures
 on $(A \to M,\varphi, \brr{\, ,\,} ,\rho ,\alpha) $, obtained as follows:
\begin{enumerate}
\item Given a hom-Gerstenhaber algebra  structure $(\Gamma(\wedge^\bullet A), \wedge, \brg{\, ,\,} ,\alpha) $, we define a bracket $\brr{\, , \,}$
on $\Gamma(A)$ by restriction of $\brg{\, ,\,}  $ to $\Gamma(A)$ and an anchor $\rho : \varphi^!A \to  \varphi^! TM $ by $\rho(X)[F]:= \brg{X,F}$
for all $X \in \Gamma(A), F \in C^{\infty}(M)$.
\item Conversely, given a hom-Lie algebroid structure $(A \to M,\varphi, [\, ,\, ],\rho,\alpha) $, we define a
hom-Gerstenhaber bracket on $ \Gamma(\wedge^\bullet A)$,
for all $X_1, \dots, X_p, Y_1, \dots, Y_q \in \Gamma(A)$, $q \geq 1$,  $F \in C^\infty (M)$, by:
$$
\brg{X_1 \wedge \dots \wedge X_p, Y_1 \wedge \dots \wedge Y_q}= \sum_{i=1}^p \sum_{j=1}^q \hbox{\small{$(-1)^{i+j}$}} [X_i, Y_j] \wedge \al(X_1 \wedge \dots \widehat{X}_i  \wedge \dots \wedge X_p \wedge Y_1 \wedge \dots  \widehat{Y}_j \wedge \dots \wedge Y_q)
$$
and by
$$
\brg{X_1 \wedge \dots \wedge X_p, F} = \sum_{i=1}^p \hbox{\small{$(-1)^{i+1}$}} \rho(X_i) [F] \wedge \al (X_1 \wedge \dots \wedge \widehat{X}_i \wedge \dots \wedge X_p).
$$
\end{enumerate}
\end{thm}

\begin{proof}
\emph{1}) We first need to justify our definition of $\rho$. It follows from the hom-Leibniz identity of the hom-Gerstenhaber algebra that
$F \mapsto \brg{X,F}$ is a $\varphi^*$-derivation and that $\brg{GX,F} = \varphi^*(G)\brg{X,F} $ for all $X \in \Gamma(A), F,G \in C^{\infty}(M)$.
Altogether, these properties imply that there is an unique vector bundle morphism $\rho: \varphi^!A \to  \varphi^! TM $
such that $\rho(X)[F]=\brg{X,F} $,  for all $X \in \Gamma(A), F \in C^{\infty}(M)$. Condition 1 in definition \ref{def:homLiealgebroid} holds by assumption.
Conditions 2 and 4 follow from proposition \ref{prop:gerstenhabergivesanchor}.
The hom-Leibniz identity of the hom-Gerstenhaber algebra:
 $$ \brg{X,FY}= \varphi^* F \brg{X,Y}+ \brg{X,F} \alpha(Y)  \quad \forall X,Y \in \Gamma(A), F \in C^\infty (M)$$
gives the condition 3 and proves the first item, since $\brg{X,F} = \rho(X)[F]$ by the very construction of $\rho$.

\noindent
\emph{2}) The hom-Leibniz identity, together with the facts that $\rho(FX)[G] = \varphi^*(F) \rho(X)[G]$ for all $F,G \in C^ \infty (M), X \in \Gamma(A)$ and that $F \mapsto \rho(X)[F]$ is a $\varphi^*$-derivation imply that $\brg{\, , \,} $ is well-defined. It also implies that it obeys to the hom-Leibniz identity. By equation (\ref{eq:Jacobiator}), it suffices to check the hom-Jacobi identity of $ \brg{\, ,\, }$
for triples made of three sections of $\Gamma(A)$ and triples made of two sections of $A$ together with one function on $M$. In the first case, it simply follows from the hom-Jacobi identity  of $\brr{\, ,\,}$ and in the second case, it is equivalent to the assumption that $(\rho,\varphi^*) $ is a representation of $
(\Gamma(A),\brr{\, ,\,},\al)$. This proves 2).

The constructions of both items are clearly inverse one to the other, and the theorem follows.
\end{proof}

The theorem shall in fact help us to construct examples of hom-Lie algebroids.

\begin{ex}
Recall \cite{McK} that a Gerstenhaber algebra structure $(\Gamma(\wedge^\bullet A ), \brg{\, ,\,} ,\wedge)$ is naturally associated to every Lie algebroid $(A,\brr{\, ,\,},\rho)$.
Given $\al:\Gamma(A) \to \Gamma(A) $ a linear endomorphism  of $\Gamma(A) $ satisfying $\alpha(FX) = \varphi^* (F) \alpha (X)$, for all $X \in \Gamma(A), F \in C^\infty(M)$ ($\varphi: M \to M$ being a given smooth map), an algebra endomorphism, again called $\al$, of
$(\Gamma(\wedge^\bullet A ),\wedge)$ can be construted as in  (\ref{eq:alphaextented}).
 This morphism $\alpha$ preserves the bracket $\brg{\, ,\,} $ provided that
$\alpha: \Gamma(A) \to \Gamma(A) $ is a Lie algebra morphism such that $\rho(\alpha(X))[\varphi^* F] = \varphi^*( \rho(X)[F])$
for all $X \in \Gamma(A), F \in C^\infty(M)$.
Example \ref{ex:GerstByCompo} then allows us to build a hom-Gerstenhaber algebra by composition, which, by theorem
\ref{theo:maintheorem}, yields a hom-Lie algebroid. Indeed, it suffices that $\brg{\, ,\,}$ satisfies the Jacobi identity on the image of $\al^2$.
For this, it suffices that  $\brg{\, ,\,}$ satisfies the Jacobi identity when applied to triples of the form $(\al^2(X),\al^2(Y),\al^2(Z))$
or $(\al^2(X),\al^2(Y),(\varphi^*)^2(F))$, with $X,Y,Z \in \Gamma(A), F \in C^\infty (M)$.
This means that $(A,\brr{\, ,\,},\rho)$ can be just assumed to be a pre-Lie algebroid (i.e. the Leibniz rule holds, but $\brr{\, ,\,}$
is not a Lie bracket) that satisfies the Jacobi identity on the image of $\al^2:\Gamma(A) \to \Gamma(A)$ and such that for all $X,Y \in \Gamma(A)$:
\begin{equation}\label{eq:conditionanchor} \rho(\brr{\al^2(X),\al^2(Y)}) - \brr{ \rho(\al^2 (X)) , \rho(\al^2 (Y)) } \end{equation}
is a vector field on $M$ that vanishes on the image of $\varphi^2 : M \to M$.
\end{ex}

We now describe two particular cases of the previous construction.

\begin{ex}
A vector field ${\mathcal V}\in \mathfrak{X}(M)$ induces a Lie algebroid as follows: $A_{\mathcal V} := M \times \mathbb{R}$ is the trivial line bundle, the Lie bracket of two sections $F,G \in C^\infty(M)=\Gamma(A_{\mathcal V})$ is given by $ [F,G] := F {\mathcal V} [G] - G {\mathcal V}[F]$
and the anchor of  $F \in C^\infty(M)=\Gamma(A_{\mathcal V})$ is the vector field $F {\mathcal V} $.
Let $\varphi: M \to M  $ be a smooth map preserving ${\mathcal V}$, i.e. $\varphi^* ({\mathcal V}[F])
= {\mathcal V}[\varphi^* F] $. Then $\varphi^*$ is a linear endomorphism of $C^\infty(M)=\Gamma(A_{\mathcal V})  $
which satisfies the required conditions to yield a hom-Lie algebroid by composition.
\end{ex}

\begin{ex}
Let $(M,\pi,\varphi)$  be a hom-Poisson manifold. Let $A = T^*M$, and let $\alpha$ be the pull-back morphism $\varphi^* : \Gamma(T^* M) \to \Gamma(T^* M)$. Then a pre-Lie algebroid structure \cite{Vaisman} is defined on $\Gamma(T^* M) $ by considering the anchor map $\rho := \pi^\# : T^*M \to TM$ together with the bracket
  $$[a,b]_{\pi} := {\mathcal L}_{\pi^\# a} b -   {\mathcal L}_{\pi^\# b} a + \diff \imath_\pi (a \wedge b).$$
It is immediate that $\varphi^*:  \Gamma(T^* M) \to \Gamma(T^* M)$ is an automorphism of this bracket and that
$\rho(\varphi^* a)[\varphi^* F] = \varphi^*( \rho(a)[F])$ for any $1$-form $a \in \Gamma(T^* M) $
and $ F \in C^\infty(M)$. Assuming the Schouten-Nijenhuis bracket $[\pi,\pi]$ to vanish on the image of $\varphi^2$ amounts to require that
$[\,,\,]_{\pi} $ satisfies the Jacobi identity when restricted to the image of $(\varphi^*)^2: \Gamma(T^* M) \to \Gamma(T^* M)$
and that condition (\ref{eq:conditionanchor}) holds. A hom-Lie algebroid can therefore be constructed by composition.
\end{ex}

\end{document}